\documentclass[12pt]{article}

\usepackage{amsmath,amsthm,amsfonts,amssymb,amscd,cite}

\allowdisplaybreaks[4]

\newtheorem{lemma}{Lemma}[section]
\newtheorem{theorem} {Theorem}[section]

\allowdisplaybreaks [4]

\headsep =
0.0in \headheight = 0.0in \topmargin = 0.3in

\begin{document}

\title{The generalized $4$-connectivity of bubble-sort graphs}

\author{Leyou Xu, Bo Zhou\footnote{Corresponding author. E-mail: zhoubo@scnu.edu.cn} \\
School of Mathematical Sciences, South China Normal University\\
Guangzhou 510631, P.R. China}

\date{}
\maketitle

\begin{abstract}
For $S\subseteq V(G)$ with $|S|\ge 2$, let $\kappa_G (S)$ denote the maximum number of internally disjoint trees connecting $S$ in $G$. For $2\le k\le n$, the generalized $k$-connectivity  $\kappa_k(G)$ of an $n$-vertex connected graph $G$ is defined to be $\kappa_k(G)=\min \{\kappa_G(S): S\in V(G) \mbox{ and } |S|=k\}$.
The generalized $k$-connectivity can serve for measuring the fault tolerance of an interconnection network.
The bubble-sort graph $B_n$ for $n\ge 2$ is a Cayley graph over the symmetric group of permutations on $[n]$
generated by transpositions from the set $\{[1,2],[2,3],\dots, [n-1,n]\}$.
In this paper, we show that for  the bubble-sort graphs $B_n$ with $n\ge 3$, $\kappa_4(B_n)=n-2$. 
\end{abstract}

\section{Introduction}

An interconnection  network is usually modelled by its topological graph, a connected graph $G$ with vertex set $V(G)$ and edge set $E(G)$, where vertices represent processors and edges represent communication links between processors.
For an interconnection network, one mainly concerns about the reliability and fault tolerance, which usually can be measured by the traditional connectivity of its topological graph.
The connectivity $\kappa(G)$ of a graph $G$ is defined to be the minimum cardinality of a subset $S\in V(G)$ such that $G-S$ is disconnected or trivial.
A graph $G$ is said to be $k$-connected  if $\kappa(G)\ge k$.
For each $2$-subset $\{x,y\}$ of vertices of $G$, let $\kappa_G(x,y)$ denote the maximum number of internally vertex disjiont $(x,y)$-paths in $G$.
A well-known theorem of Whitney \cite{Whitney} says that
$\kappa(G)=\min\{\kappa_G(x,y): \{x,y\}\subseteq V(G)\}$.

For a set $S$ of vertices in a connected graph $G$ and  trees  $T_1, \dots, T_{\ell}$ in $G$, we say $T_1, \dots, T_{\ell}$ are  $\ell$ internally edge disjoint trees connecting $S$ in $G$  if these trees are pairwise edge disjoint and $V(T_i)\cap V(T_j)=S$ for every pair $i,j$ of distinct integers with $1\le i,j\le \ell$.

Chartrand et al. \cite{Chartrand} and Hager \cite{Hager} proposed the concept of the generalized $k$-connectivity of an $n$-vertex graph $G$ for $k=2, \dots, n$, see also \cite{H2, COZ}.
For any set $S$ of vertices of  $G$ with $|S|\ge 2$, the generalized connectivity of $S$, written as $\kappa_G(S)$, is the maximum number of internally disjoint trees  connecting $S$ in $G$.
For $2\le k\le |V(G)|$, the generalized $k$-connectivity (or $k$-tree connectivity) of $G$, $\kappa_k(G)$, is the minimum value for $\kappa_G(S)$ over all subsets $S$ of vertices with $|S|=k$.
Note that $\kappa_2(G)$ is the connectivity of $G$, and $\kappa_n(G)$ is the maximum number of edge disjoint spanning trees contained in $G$ \cite{Na,Tu} (or the spanning tree packing number of $G$ \cite{Pa}).
The generalized $k$-connectivity has been used to measure the capability of a network  to connect any $k$ vertices.


Cayley  graphs have been used extensively to design interconnection networks. The Cayley graph  $\mbox{Cay}(X,S)$, where $X$ is a group with identity $e$, $e\not\in S\subseteq X$ and $S$ is closed under inversion, is the graph with vertex set $X$, such that $g$ and $h$ for $g,h\in X$ are adjacent if and only if $h=gs$ for some $s\in S$.

Denote $\mbox{Sym}(n)$ the symmetric group (i.e., the group of all permutations)  on $[n]
=\{1,\dots,n \}$. For convenience, we use $(p_1,\dots,p_n)$ to denote the permutation $\sigma$ such that
$\sigma(i)=p_i$ for $i\in [n]$,
and $[i,j]$ with $1\le i<j\le n$ to denote the permutation $(1, \dots, i-1, j, i+1, \dots, j-1, i, j+1, \dots, n)$,
which is called a transposition.
The composition $\sigma \pi$ of permutations $\sigma$ and $\pi$ is  the function that maps any element $i\in [n]$ to $\sigma(\pi(i))$.
Thus
\[(p_1,\dots, p_i,\dots, p_j,\dots, p_n)[i,j]=(p_1,\dots, p_j,\dots, p_i,\dots, p_n),
\]
 which swaps the objects at positions $i$ and $j$.

Let $\mathcal{T}$ be a set of transpositions from  $[n]$.
The (transposition generating) graph  of $\mathcal{T}$, denoted by $G_\mathcal{T}$, is the graph with vertex set $[n]$ such that, for $i,j\in [n]$,
vertices $i$ and $j$ are adjacent  if and only if $[i,j]\in \mathcal{T}$.
It is known that the Cayley graph  $\mbox{Cay}(\mbox{Sym}(n),\mathcal{T})$ is connected if and only if  $G_\mathcal{T}$ is connected.
If $G_\mathcal{T}$ is the star, then $\mbox{Cay}(\mbox{Sym}(n),\mathcal{T})$ is called a star graph, denoted by  $S_n$.
If  $G_\mathcal{T}$ is the path, then $\mbox{Cay}(\mbox{Sym}(n),\mathcal{T})$ is called a bubble-sort graph, denoted by $B_n$.
Observe that $B_2$ is the $2$-vertex complete graph and $B_3$ is the $6$-vertex cycle.
Generally, $B_n$ is an $n!$-vertex  bipartite, vertex transitive and regular graph of degree $n-1$.

The generalized connectivity has been studied extensively, see the recent book \cite{LMb}.
There has been lots of results on the generalized $3$-connectivity for various classes of graphs, see, e.g., \cite{AC,GLW,LST,LTY,LM,ZH0,ZHW}.
For example, Li et al. \cite{LTY} showed that $\kappa_3(S_n)=\kappa_3(B_n)=n-2$ for $n\ge 3$.
The generalized $4$-connectivity has also received attention, see \cite{LiC,LZ,ZH1,ZH2}. Li et al. \cite{LiC} showed that
 $\kappa_4(S_n)=n-2$ for $n\ge 3$. More closely related results may be found, see, e.g.,  \cite{CLLM,LLS,LLSS}.

In this paper, we will determine the generalized $4$-connectivity of the bubble-sort graph $B_n$. We show the following result.

\begin{theorem}\label{1}
For $n\ge 3$, $\kappa_4(B_n)=n-2$.
\end{theorem}

\section{Preliminaries}

For $v\in V(G)$, denote by $N_G(v)$ the set of neighbors of $v$ in $G$, $\delta_{G}(v)=|N_G(v)|$ and $N_G[v]=N_G(v)\cup \{v\}$.
For a subset $S\subseteq V(G)$, denote by $G[S]$ the subgraph of $G$ induced by $S$.

For $x,y\in V(G)$, a path joining $x$ and $y$ in $G$ is called an $(x,y)$-path. For $X,Y\subset V(G)$, an $(X,Y)$-path is a path joining $x$ and $y$ in $G$ for some  $x\in X$ and some $y\in Y$, and any other vertex of  the path  (if any exists) are not in $X\cup Y$.
We write $(x, Y)$-path instead of $(\{x\}, Y )$-path.

\begin{lemma}\label{menger}\cite{BM}
Let $G$ be a $k$-connected graph, and let $X,Y\subset V(G)$ with $|X|, |Y|\ge k$.
Then there are $k$ pairwise vertex disjoint $(X,Y)$-paths in $G$.
\end{lemma}

\begin{lemma}\label{menger1}\cite{BM}
Let $G$ be a $k$-connected graph, and let $x\in V(G)$ and $Y\subset V(G)\setminus\{ x\}$ with $|Y|\ge k$.
Then there are $k$ internally vertex disjoint $(x,Y)$-paths such that $x$ is the only common terminal vertex.
\end{lemma}

The following lemma tells us an upper bound on $\kappa_{k}(G)$ for a graph $G$.

\begin{lemma}\cite{LiSS2}\label{delta}
Let $G$ be a connected graph with minimum degree $\delta$.
Then $\kappa_k(G)\le \delta$ for $3\le k\le |V(G)|$.
Furthermore, if there exist two adjacent vertices of degree $\delta$ in $G$, then $\kappa_k(G)\le \delta-1$.
\end{lemma}

\begin{lemma}\label{kappa}\cite{EC}
$\kappa(B_n)=n-1$ for $n\ge 2$.
\end{lemma}

\begin{lemma}\label{kappa3}\cite{LTY}
$\kappa_3(B_n)=n-2$ for $n\ge 3$.
\end{lemma}

As we consider the bubble-sort graph $B_n$, we may suppose without loss of generality that $\mathcal{T}=\{[i,i+1]: i\in [n-1]\}$. Then
$E(G_\mathcal{T})=\{i(i+1): i\in [n-1]\}$.

For $i\in [n]$, let $\mbox{Sym}_i(n)$ denote the set of all permutations of $[n]\setminus\{i\}$.
For $\sigma=(p_1,\dots,p_{n-1})\in \mbox{Sym}_i(n)$,
we have $\sigma(j)=p_j$ for $j<i$ and $\sigma(j)=p_{j-1}$ for $j>i$.
Let
\[
V_i=\{(p_1,\dots,p_{n-1},i): (p_1,\dots,p_{n-1})\in \mbox{Sym}_i(n)\}
\]
and $B_{n-1}^i=B_n[V_i]$ for $i\in [n]$.
Then $V(B_n)$ can be partitioned  into $V_1, \dots, V_n$ and   $B_{n-1}^i\cong B_{n-1}$ for $i\in [n]$.
We call $B_{n-1}^1,\dots,B_{n-1}^n$ the main parts of $B_n$.

If  $u=(p_1,\dots,p_{n-1},k)\in V_k$, then  $u$ is in the main part $B_{n-1}^k$.
Let $u_i=u[i,i+1]$ for $i\in [n-1]$.
Then $N_{B_n}(u)=\{u_i: i\in [n-1]\}$ with
$u_1,\dots,u_{n-2}\in V_k$ and $u_{n-1}\in V_{p_{n-1}}$.
Note that $u_{n-1}$ is the unique neighbor of $u$ outside $B_{n-1}^k$, which we call  the out-neighbor of $u$, written as $u'$ throughout this paper.
The other $n-2$ neighbors of $u$ are called the in-neighbors of $u$.
The out-neighbor of $u_i$ is $u_i'=u_i[n-1,n]$ for $i\in [n-1]$.
Then $u_i'\in V_{p_{n-1}}$ for $i\in [n-3]$ and $u_{n-2}'\in V_{p_{n-2}}$.
Note that  $u_{n-1}'=u$.

It can be verified that any two distinct vertices have different out-neighbors and $|((\cup_{u\in V_i} N_{B_n}(u))\setminus V_i)\cap V_j|=(n-2)!$ for $i,j\in [n]$ with $i\neq j$, see \cite{EC}.

For $\{i,j\}\subset [n]$ with  $n\ge 3$, it is shown in \cite{LTY} that
\[
\kappa(B_n[V_i\cup V_j])=n-2.
\]
By the proof in \cite{LTY},  there are $n-2$ internally vertex disjoint paths between any two vertices in
$B_n[V_i\cup V_j]$. So we have the following result.

\begin{lemma}\label{ijout}
Let $B_{n-1}^1, \dots, B_{n-1}^n$ be the main parts of $B_n$, where $n\ge 3$.
For any $\emptyset\ne I\subset [n]$,
\[
\kappa(B_n[\cup_{i\in I}V_i])=n-2.
\]
\end{lemma}

Suppose that $T_1, \dots, T_s$ are $s\ge 2$ trees such that $|V(T_i)\cap V(T_j)|=0,1$ for any $i,j$ with $1\le i<j\le s$. If
the graph with vertex set $\cup_{i=1}^sV(T_i)$ and edge set $\cup_{i=1}^sE(T_i)$ connected, then it is a tree, denoted by $T_1+\dots +T_s$. It is possible that $T_i$ is  a path.

Fix $i\in [n]$. For $j\in [n]\setminus \{i\}$, let
\[
V_j^i=\{(p_1,\dots,p_{n-2},j,i): (p_1,\dots,p_{n-2})\in\mbox{Sym}_{i,j}(n)\},
\]
where $\mbox{Sym}_{i,j}(n)$ denotes the set of permutations of $[n]\setminus\{i,j\}$.
Denote the induced subgraph $B_n[V_j^i]$ by $B_{n-2}^{(i,j)}$.

\section{Proof of Theorem \ref{1}}

\begin{proof}[Proof of Theorem \ref{1}]
By Lemma \ref{delta} and the fact that $B_n$ is $(n-1)$-regular, we have  $\kappa_4(B_n)\le n-2$.
So it suffices to show that $\kappa_4(B_n)\ge n-2$.
Let $S$ be an arbitrary subset of $V(B_n)$ with $|S|=4$, say
 $S=\{x,y,z,w\}$.
Then, it suffices to show that
\[
\mbox{there are $n-2$ internally edge disjoint trees connecting $S$ in $B_n$}.
\]
We prove this statement by induction on $n$.

If $n=3$, it is evident that  there exists a tree containing vertices in $S$, so the statement is true.
Suppose that $n \ge 4$ and the statement is true for $B_{n-1}$.

Recall that $B_{n-1}^1, \dots, B_{n-1}^n$ are the main parts of $B_n$.
We consider the following five cases separately in subsections 3.1--3.5:
\begin{itemize}
\item {\bf Case 1.}  The four vertices of $S$ lie in a main part of $B_n$;

\item {\bf Case 2.}  Two vertices of $S$ lie in a main part and the other two vertices in $S$ lie in another main part of $B_n$;

\item {\bf Case 3.}  The four vertices of $S$ lie in three different  main parts of $B_n$;

\item {\bf Case 4.} The four vertices of $S$ lie in four different  main parts of $B_n$;

\item {\bf Case 5.} Three vertices of $S$ lie in a main part and the remaining one lies in another main part of $B_n$.
\end{itemize}

\subsection{Case 1}

Assume that $x,y,z,w$ are in $B_{n-1}^1$.
Note that $B_{n-1}^1\cong B_{n-1}$.
By the induction hypothesis, there are $n-3$ internally edge disjoint trees $T_1, \dots, T_{n-3}$ connecting $S$ in the main part $B_{n-1}^1$ of $B_n$.
By Lemma \ref{ijout}, $B_n[V(B_n)\setminus V_1]$ is connected, so there is a spanning tree $T$ in $B_n[V(B_n)\setminus V_1]$. Note that $x',y',z',w'$ are distinct four vertices in $B_n[V(B_n)\setminus V_1]$.
So  $T_{n-2}=T+xx'+yy'+zz'+ww'$  is a tree containing vertices in  $S$ and $V(T_{n-2})\cap V_1=S$.
It thus follows that  $T_1, \dots, T_{n-2}$ are $n-2$ internally edge disjoint trees connecting $S$ in $B_{n}$.

\subsection{Case 2}

Assume that $x,y\in V(B_{n-1}^1)$ and $z,w \in V(B_{n-1}^2)$.
By Lemma \ref{kappa},  $\kappa(B_{n-1}^2)=\kappa(B_{n-1})=n-2$, so there are $n-2$  internally vertex disjoint $(z,w)$-paths $Q_1, \dots, Q_{n-2}$ in $B_{n-1}^2$. Since $|N_{B_{n-1}^2}(z)|=n-2$ and $Q_1, \dots, Q_{n-2}$ are internally vertex disjoint $(z,w)$-paths, we may assume that $V(Q_i)\cap N_{B_{n-1}^2}(z)=\{z_i\}$ for  $i\in [n-2]$.

\noindent
{\bf Case  2.1.} One of $x'$ and $y'$, say $x'$, is not in $B_{n-1}^2$,  and one of $z'$ and $w'$, say $z'$, is not in $B_{n-1}^1$.

By Lemma \ref{kappa}, there are $n-2$ internally vertex disjoint $(x,y)$-paths  $L_1$, $\dots, L_{n-2}$, and 
 we may assume that $V(L_i)\cap N_{B_{n-1}^1}(x)=\{x_i\}$ for $i\in [n-2]$.

Note that  there is exactly one in-neighbor of $x$, say $x_{n-2}$, whose out-neighbor and $x'$ lie in different main parts, and
there is exactly one in-neighbor of $z$, say $z_{n-2}$, whose out-neighbor and $z'$ lie in different main parts.
Let
\[
X=\{x_i': i\in [n-3]\}\cup \{x'\} \mbox{ and } Z=\{z_i': i\in [n-3]\}\cup \{z'\}.
\]
It is evident that $|X|=|Z|=n-2$. By Lemmas \ref{menger} and \ref{ijout}, there are $n-2$ disjoint $(X,Z)$-paths $R_1,\dots,R_{n-2}$ in $B_n[\cup_{i=3}^{n} V_i]$.
Suppose that $x'\in V(R_{n-2})$, $x_i'\in V(R_i)$ for $i\in [n-3]$, $z'\in V(R_s)$ for some $s\in [n-2]$, $z_i'\in V(R_i)$ for $i\in [n-3]\setminus\{s\}$ and $z_s'\in V(R_{n-2})$.
If $s=n-2$, let
\[
T_i=L_i+x_ix_i'+R_i+z_i'z_i+Q_i\mbox{ for } i\in[n-3]
\]
and
\[
T_{n-2}=L_{n-2}+xx'+R_{n-2}+z'z+Q_{n-2}.
\]
Otherwise, let
\[
T_i=L_i+x_ix_i'+R_i+z_i'z_i+Q_i\mbox{ for } i\in [n-3]\setminus\{s\},
\]
\[
T_s=L_s+x_sx_s'+R_{s}+z'z+Q_{n-2}
\]
and
\[
T_{n-2}=L_{n-2}+xx'+R_{n-2}+z_s'z_s+Q_s.
\]
Then it is easy to see that $T_1,\dots,T_{n-2}$ are $n-2$ internally edge disjoint trees connecting $S$.

\noindent
{\bf Case  2.2.} $x'$ and $y'$ are both in $B_{n-1}^2$ and one of $z'$ and $w'$ 
is not in $B_{n-1}^1$, or $z'$ and $w'$ are both in $B_{n-1}^1$ and one of $x'$ and $y'$ is not in $B_{n-1}^2$.

Assume that $x'$ and $y'$ are both in $B_{n-1}^2$ and one of $z'$ and $w'$, say $z'$, is not in $B_{n-1}^1$.

Suppose that $n=4$. Then $x$ and $y$ are adjacent.
If $w'\in V_1$, then $w'=x$ or $w'=y$, say $w'=y$.
Let $x_1=x[2,3]$ and $y_1=y[2,3]$.
Then $x_1',y_1',z'\in V_3\cup V_4$. As $B_4[V_3\cup V_4]$ is connected,
there is a tree $T_1$ containing $x_1',y_1',z'$.
Let
\[
T_1^*=xx_1+x_1x_1'+yy_1+y_1y_1'+T_1+z'z+Q_1
\]
and
\[
T_2=xy+yw+Q_2.
\]
Then $T_1^*$ and $T_2$ are two internally edge disjoint trees connecting $S$.
Otherwise, $w'\notin V_1$.
Since $x',y'\in V_2$, there is a tree $F_1$  in $B_3^2$ containing $x',y',z,w$.
Similarly, there is a tree $F_2$  in $B_4[V_1\cup V_3\cup V_4]$ containing $\{x,y,z',w'\}$.
Then $F_1^*=F_1+x'x+y'y$ and
$F_2^*=F_2+w'w+z'z$  are  two internally edge disjoint trees connecting $S$.

Suppose that $n\ge 5$.
For $j=2,\dots,n$, let
\[
V_j^1=\{(p_1,\dots,p_{n-2},j,1): (p_1,\dots,p_{n-2})\in\mbox{Sym}_{1,j}(n)\},
\]
where $\mbox{Sym}_{1,j}(n)$ denotes the set of permutations of $[n]\setminus\{1,j\}$.
Denote the induced subgraph $B_n[V_j^1]$ by $B_{n-2}^{(j)}$.

Since $B_{n-2}^{(1,j)}\cong B_{n-2}$ and $B_{n-1}^1\cong B_{n-1}$, we view  $B_{n-2}^{(1,2)},\dots,B_{n-2}^{(1,n)}$ as the main parts of $B_{n-1}^1$.
Then  $x$ and $y$ are in $B_{n-2}^{(1,2)}$.
By Lemma \ref{kappa}, $\kappa(B_{n-2}^{(1,2)})=n-3$, so there exist $n-3$ internally disjoint $(x,y)$-paths $L_1,L_2,\dots,  L_{n-3}$ in $B_{n-2}^{(1,2)}$.
Note that there are $n-3$ vertices adjacent to $x$ in $B_{n-2}^{(1,2)}$.
Then each $L_i$ contains exactly one vertex in $N_{B_{n-2}^{(1,2)}}(x)$, which we denote by $x_i$, where $i\in [n-3]$.

Assume that $z_{n-2}$ is the vertex whose out-neighbor is not in the same main part as $z'$.

Let $x=(p_1,\dots, p_{n-2},2,1)$ and $y=(r_1,\dots, r_{n-2},2,1)$.
Let $x_{n-2}=x[n-2,n-1]$, $x_{n-2,1}=x_{n-2}[n-3,n-2]$, $x_{n-2,2}=x_{n-2,1}[n-4,n-3]$, $x_{n-2,3}=x_{n-2,2}[n-3,n-2]$ and $\widehat{x}_{n-2}=x_{n-2,3}[n-2,n-1]$. That is,
\[
x_{n-2}=(p_1,\dots,p_{n-4}, p_{n-3},2, p_{n-2},1),
\]
\[
x_{n-2,1}=(p_1,\dots, p_{n-4},2,p_{n-3}, p_{n-2},1),
\]
\[
x_{n-2,2}=(p_1,\dots,2,p_{n-4},p_{n-3},p_{n-2},1),
\]
\[
x_{n-2,3}=(p_1,\dots,2,p_{n-3},p_{n-4},p_{n-2},1)
\]
and
\[
\widehat{x}_{n-2}=(p_1,\dots, 2,p_{n-3},p_{n-2}, p_{n-4},1).
\]
Let
\[
P_x=xx_{n-2}x_{n-2,1}x_{n-2,2}x_{n-2,3}\widehat{x}_{n-2}.
\]
There are three probabilities:
(i) If $\{r_{n-3},r_{n-2}\}=\{p_{n-3},p_{n-2}\}$, then set
$y_{n-2}=y[n-2,n-1]$,
$y_{n-2,1}=y_{n-2}[n-3,n-2]$,
$y_{n-2,2}=y_{n-2,1}[n-4,n-3]$,
$y_{n-2,3}=y_{n-2,2}[n-3,n-2]$,
$\widehat{y}_{n-2}=y_{n-2,3}[n-2,n-1]$
and
$P_y=yy_{n-2}y_{n-2,1}y_{n-2,2}y_{n-2,3}\widehat{y}_{n-2}$.
(ii) If $r_{n-2}\in \{p_{n-3},p_{n-2}\}$ and $r_{n-3}\notin \{p_{n-3},p_{n-2}\}$, then set
$y_{n-2}=y[n-2,n-1]$, $y_{n-2,1}=y_{n-2}[n-3,n-2]$, $\widehat{y}_{n-2}=y_{n-2,1}[n-2,n-1]$
and
$P_y=yy_{n-2}y_{n-2,1}\widehat{y}_{n-2}$.
(iii) Otherwise, set
$\widehat{y}_{n-2}=y[n-2,n-1]$ and $P_y=y\widehat{y}_{n-2}$.
As $x\neq y$, we have $V(P_x)\cap V(P_y)=\emptyset$.
Correspondingly to (i)--(iii), we have by Lemma \ref{ijout} that each of  $B_{n-1}^1[V_{p_{n-4}}^1\cup V_{r_{n-4}}^1]$,
 $B_{n-1}^1[V_{p_{n-4}}^1\cup V_{r_{n-3}}^1]$,
or
$B_{n-1}^1[V_{p_{n-4}}^1\cup V_{r_{n-2}}^1]$
is connected, so there is a $(\widehat{x}_{n-2},\widehat{y}_{n-2})$-path $P_{xy}$ in one of them.
Let
\[L_{n-2}=P_x+P_{xy}+P_y.
\]
Since $V(L_{n-2})\cap V_2^1=\{x,y\}$, we have $n-2$ internally disjoint $(x,y)$-paths in $B_{n-1}^1$.

\noindent
{\bf Case  2.2.1.} $x,y$ are not adjacent.

Let $\widehat{x}_i=x_i[n-2,n-1]$ for $i\in [n-3]$.
Then
$|\{\widehat{x}_i': i\in [n-3] \}\cap V_{p_{n-2}}^1|=n-4$ and $|\{\widehat{x}_i': i\in [n-3] \}\cap V_{p_{n-3}}^1|=1$.
Since $x_i\in N_{B_n}(x)$ for $1\le i\le n-3$, we have  $\widehat{x}_{i}\neq \widehat{x}_{j}$ if $i\neq j$.
Note that  $xy\notin E(B_n)$.
By comparing the position of `$2$' in the  permutation corresponding to  the vertices on $P_x$  and in $\widehat{x}_{i}$ for $i\in [n-3]$, we have
$V(P_x)\cap \{ \widehat{x}_{i}: i\in [n-3]\}=\emptyset$. Similarly, $V(P_y)\cap \{ \widehat{x}_{i}: i\in [n-3]\}=\emptyset$.

Let $X=\{\widehat{x}_i': i\in [n-2] \}$, and $Z=\{z_i': i\in [n-3]\}\cup \{z'\}$.
Note that $X\subseteq \cup_{i=3}^n V_i$ and $Z\subseteq \cup_{i=3}^n V_i$.
By Lemmas \ref{menger} and \ref{ijout}, there are $n-2$ disjoint $(X,Z)$-paths $R_1,\dots,R_{n-2}$ in $B_n[\cup_{i=3}^n V_i]$.
Assume that $z'\in V(R_{n-2})$, $z_i'\in V(R_i)$ for $i\in [n-3]$,  $\widehat{x}_{n-2}'\in V(R_s)$ for some $s\in [n-2]$, $\widehat{x}_i'\in V(R_i)$ for $i\in[n-3]\setminus\{s\}$ and $\widehat{x}_s'\in V(R_{n-2})$.
If $s=n-2$, let
\[
T_i=L_i+x_i\widehat{x}_i+\widehat{x}_i\widehat{x}_i'+R_i+z_i'z_i+Q_i \mbox{ for } i\in [n-3]
\]
and
\[
T_{n-2}=L_{n-2}+\widehat{x}_{n-2}\widehat{x}_{n-2}'+R_{n-2}+z'z+Q_{n-2}.
\]
Otherwise, let
\[
T_i=L_i+x_i\widehat{x}_i+\widehat{x}_i\widehat{x}_i'+R_i+z_i'z_i+Q_i \mbox{ for } i\in[n-3]\setminus\{s\},
\]
\[
T_s=L_{n-2}+\widehat{x}_{n-2}\widehat{x}'_{n-2}+R_s+z_s'z_s+Q_s
\]
and
\[
T_{n-2}=L_s+x_s\widehat{x}_s+\widehat{x}_s\widehat{x}_s'+R_{n-2}+z'z+Q_{n-2}.
\]
Then $T_1, \dots, T_{n-2}$  are  $n-2$ internally edge disjoint trees  connecting $S$ in $B_n$.

\noindent
{\bf Case  2.2.2.} $x,y$ are adjacent.

Assume that $L_1=xy$.
Let $\widehat{x}_i=x_i[n-2,n-1]$ for  $i=2,\dots,n-3$.
By similar argument as in  Case 2.2.1, we have $V(P_x)\cap \{ \widehat{x}_i: i=2,\dots,n-3\}=\emptyset$ and $V(P_y)\cap \{ \widehat{x}_i: i=2,\dots,n-3\}=\emptyset$.

Suppose that $N_{B_{n-1}^2}[x']\cap (\cup_{i=1}^{n-2} V(Q_i))=\emptyset$. Let $\widehat{x}_1=x'[n-2,n-1]$.
Let $X$ and $Z$ be defined the same as that in Case 2.2.1.
Then there are $n-2$ internally vertex disjoint $(X,Z)$-paths $R_i$ in $B_n[\cup_{i=3}^n V_i]$ for $i\in [n-2]$.
If $s\neq 1$, let $T_i$ be  defined as in Case 2.2.1 for $i=2,\dots,n-2$, and
let
\[T_1=xy+xx'+x'\widehat{x}_1+\widehat{x}_1\widehat{x}_1'+R_1+z_1'z_1+Q_1.
\]
Otherwise, let $T_i$ be defined as in Case 2.2.1 for $i=2,\dots,n-3$,
\[
T_1=L_{n-2}+\widehat{x}_{n-2}'\widehat{x}_{n-2}'+R_1+z_1'z_1+Q_1
\]
and
\[
T_{n-2}=xy+xx'+x'\widehat{x}_1+\widehat{x}_1\widehat{x}_1'+R_{n-2}+z'z+Q_{n-2}.
\]
In either case,  there are $n-2$ internally edge disjoint trees connecting $S$.

Otherwise,
assume that  $\widehat{x}_1\in N_{B_{n-1}^2}[x']\cap V(Q_\ell)$ for some $\ell\in [n-2]$.
So
\[
T_1=
\begin{cases}
Q_\ell+x'\widehat{x}_1+xx'+xy & \mbox{if $\widehat{x}_1\ne x'$}\\
Q_\ell+xx'+xy & \mbox{otherwise}
\end{cases}
\]
is a tree containing vertices in $S$.
Let $X=\{\widehat{x}_i': i=2,\dots,n-2 \}$, $Z=\{z_i': i\in [n-3]\}$ if $\ell=n-2$ and $Z=\{z_i': i\in[n-3]\setminus\{\ell\}\}\cup \{z'\}$ otherwise.
By Lemmas \ref{menger} and \ref{ijout}, there are $n-3$ internally vertex disjoint $(X,Z)$-paths $R_1,\dots,R_{n-3}$ in $B_n[\cup_{i=3}^n V_i]$.
Assume that $\widehat{x}_{i+1}',z_i'\in V(R_i)$ for $i\in [n-3]$ if $\ell=n-2$.
Let
\[
T_i=L_i+x_i\widehat{x}_i+\widehat{x}_i\widehat{x}_i'+R_{i-1}+z_{i-1}'z_{i-1}+Q_{i-1}\mbox{ for } i=2,\dots,n-3,
\]
and
\[
T_{n-2}=L_{n-2}+\widehat{x}_{n-2}\widehat{x}_{n-2}'+R_{n-3}+z_{n-3}'z_{n-3}+Q_{n-3}.
\]
Otherwise, we may suppose without loss of generality that $\ell=1$.
Assume that $z'\in V(R_{n-3})$, $z_i'\in V(R_{i-1})$ for $i=2,\dots,n-3$, $\widehat{x}_s'\in V(R_{n-3})$, $\widehat{x}_{n-2}'\in  V(R_{s-1})$ and $\widehat{x}_i'\in V(R_{i-1})$ for $i\in [n-3]\setminus\{1,s\}$.
If $s=n-2$, let
\[
T_i=L_i+x_i\widehat{x}_i+\widehat{x}_i\widehat{x}_i'+R_{i-1}+z_i'z_i+Q_i\mbox{ for } i=2,\dots,n-3,
\]
and
\[
T_{n-2}=L_{n-2}+\widehat{x}_{n-2}\widehat{x}_{n-2}'+R_{n-3}+z'z+Q_{n-2}.
\]
Otherwise, let
\[
T_i=L_i+x_i\widehat{x}_i+\widehat{x}_i\widehat{x}_i'+R_{i-1}+z_i'z_i+Q_i\mbox{ for } i\in[n-3]\setminus\{1,s\},
\]
\[
T_s=L_{n-2}+\widehat{x}_{n-2}\widehat{x}_{n-2}'+R_{s-1}+z_s'z_s+Q_s
\]
and
\[
T_{n-2}=L_s+x_s\widehat{x}_s+\widehat{x}_s\widehat{x}_s'+R_{n-3}+z'z+Q_{n-2}.
\]
Then  $T_1, \dots, T_{n-2}$ are $n-2$ internally edge disjoint trees connecting $S$ in $B_n$.

\noindent
{\bf Case  2.3.} Both $x',y'$ are in $B_{n-1}^2$ and $z',w'$ are in $B_{n-1}^1$.

If $n=4$, then $B_4[S]$ is a cycle of length four with edges $xy,zw,xz,yw$.
Let $x_1=x[2,3]$, $y_1=y[2,3]$ and $z_1=z[2,3]$.
Then $x_1',y_1',z_1'\in V_3\cup V_4$. So there is a tree $T_1'$ connecting $x_1',y_1',z_1'$.
Then
\[
T_1=xx_1+x_1x_1'+yy_1+y_1y_1'+T_1'+z_1'z_1+z_1z+zw
\]
and
\[
T_2=zx+xy+yw
\]
are two internally edge disjoint trees connecting $S$.

For $n\ge 5$, by the same way as in Case 2.2,  we may construct $n-2$ internally vertex disjoint $(x,y)$-paths
in $B_{n-1}^1$, and $n-2$ internally vertex disjoint $(z,w)$-paths
in $B_{n-1}^2$, and so we may obtain $n-2$ internally edge disjoint trees connecting $S$.

\subsection{Case 3}

Assume that $x,y\in V_1$, $z\in V_2$ and  $w\in V_3$.
Let \[x=(p_1,\dots,p_{n-1},1) \mbox{ and } y=(r_1,\dots,r_{n-1},1).\]
Then $x'\in V_{p_{n-1}}$ and $y'\in V_{r_{n-1}}$.
By considering whether the out-neighbors of $x$ and $y$ are in the same main part of $B_n$, we discuss the following two cases.

\noindent
{\bf Case  3.1.} $x'$ and $y'$ are in the different main parts, i.e., $p_{n-1}\neq r_{n-1}$.

Since $\kappa(B_{n-1}^1)=n-2$, there are $n-2$ internally vertex disjoint $(x,y)$-paths $L_1,\dots,L_{n-2}$ in $B_{n-1}^1$.
Let $\widehat{x}=x[n-2,n-1]$ and $\widehat{y}=y[n-2,n-1]$.
Note that each $L_i$ contains exactly one vertex in $N_{B_{n-1}^1}(x)$ and exactly one vertex in $N_{B_{n-1}^1}(y)$ for $i\in [n-2]$. Assume that $\widehat{x}\in V(L_{n-2})$ and $\widehat{y}\in V(L_s)$ for some $s\in [n-2]$.
Assume that $V(L_i)\cap N_{B_n}(x)=\{x_i\}$ for $i\in [n-3]$ and  $V(L_i)\cap N_{B_n}(y)=\{y_i\}$ for $i\in [n-2]\setminus \{s\}$.
Let $X=\{x_i': i\in [n-3] \}\cup\{x'\}$ and $Y=\{y_i': i\in[n-2]\setminus\{s\} \}\cup\{y'\}$.

Assume that $p_{n-1}\neq 3$ and $r_{n-1}\neq 2$, otherwise, we change the role of $x$ and $y$ in the following proof.
By Lemmas \ref{menger1} and \ref{ijout}, there are $n-2$ internally vertex disjoint $(z,X)$-paths $Q_1,\dots,Q_{n-2}$ in $B_n[V_2\cup V_{p_{n-1}}]$ and $n-2$ internally vertex disjoint $(w,Y)$-paths $R_1,\dots,R_{n-2}$ in $B_n[V_3\cup V_{r_{n-1}}]$.
Assume that $x'\in V(Q_{n-2})$, $x_i'\in V(Q_i)$ for $i\in [n-3]$, and $y'\in V(R_s)$ and $y_i'\in V(R_i)$ for $i\in[n-2]\setminus\{s\}$.
If $s=n-2$, let
\[
T_i=Q_i+x_i'x_i+L_i+y_iy_i'+R_i\mbox{ for } i\in [n-3],
\]
and
\[
T_{n-2}=Q_{n-2}+x'x+L_{n-2}+yy'+R_{n-2}.
\]
Otherwise, let
\[
T_i=Q_i+x_i'x_i+L_i+y_iy_i'+R_i\mbox{ for }  i\in[n-3]\setminus\{s\},
\]
\[
T_s=Q_s+x_s'x_s+L_s+y'y+R_s,
\]
and
\[
T_{n-2}=Q_{n-2}+xx'+L_{n-2}+y_{n-2}y_{n-2}'+R_{n-2}.
\]
Then $T_1,\dots, T_{n-2}$ are $n-2$ internally disjoint edge disjoint trees connecting $S$.

\noindent
{\bf Case  3.2.} $x'$ and $y'$ are in the same main part, i.e., $p_{n-1}= r_{n-1}$.

Assume that $p_{n-1}\neq 3$.
By similar argument as in Case 2.2, we obtain $n-2$ internally vertex disjoint $(x,y)$-paths $L_1,\dots,L_{n-2}$.
Let $x_i,\widehat{x}_i$ for $i\in [n-2]$ and $X$ be defined the same way as in Case 2.2.
Suppose that $V(L_i)\cap N_{B_n}(y)=\{y_i\}$ for $i\in [n-3]$.
Let $Y=\{y_i: i\in [n-3]\}\cup \{y'\}$.
By Lemmas \ref{menger1} and \ref{ijout},  there are $n-2$ internally disjoint $(w,X)$-paths $Q_1,\dots,Q_{n-2}$ in $B_n[V(B_n)\setminus(V_1\cup V_2\cup V_{p_{n-1}})]$ and there are $n-2$ internally disjoint $(z,Y)$-paths $R_1,\dots,R_{n-2}$ in $B_n[V_2\cup V_{p_{n-1}}]$.
Assume that $\widehat{x}_i'\in V(Q_i)$ for $i\in [n-2]$, $y_i'\in V(R_i)$ for $i\in [n-3]$ and $y'\in V(R_{n-2})$.
Let
\[
T_i=Q_i+\widehat{x}_i'\widehat{x}_i+x_i\widehat{x}_i+L_i+y_iy_i'+R_i\mbox{ for }i\in [n-3]
\]
and
\[
T_{n-2}=Q_{n-2}+\widehat{x}_{n-2}'\widehat{x}_{n-2}+L_{n-2}+yy'+R_{n-2}.
\]
Then there are $n-2$ internally edge disjoint trees $T_1,\dots,T_{n-2}$ connecting $S$.

\subsection{Case 4}

Assume that $x\in V_1$, $y\in V_2$, $z\in V_3$ and $w\in V_4$.
Suppose first that there are at least two vertices in $S$ whose out-neighbors  lie in $\cup_{i=5}^n V_i$, say $x',y'\in \cup_{i=5}^n V_i$.
By Lemma and \ref{ijout}, there are $n-2$ internally vertex disjoint $(x,z)$-paths $L_1,\dots,L_{n-2}$ in $B_n[V_1\cup V_3]$ and $n-2$ internally vertex disjoint $(y,w)$-paths $Q_1,\dots,Q_{n-2}$ in $B_n[V_2\cup V_4]$.
Then by similar argument as in  Case 2.1, we can obtain $n-2$ internally edge disjoint trees connecting $S$.

Suppose next that there is at most one vertex in $S$ whose unique out-neighbor lies in $\cup_{i=5}^n V_i$, that is, there are three vertices in $S$, say $x,y,z$, with $x',y',z'\in \cup_{i=1}^4V_i$.

Note that $x'\not\in  V_1$. Assume that $x'\in V_2$  (if $x'\in V_3$ or $x'\in V_4$, the argument is similar by viewing $z$ or $w$ as $y$).
We consider the following two cases.

\noindent
{\bf  Case 4.1.} $y'\in V_1$.

Recall that $z'\in V_1\cup V_2\cup V_4$.
Suppose first that $z'\in V_4$. By Lemma \ref{ijout}, there are $n-2$ internally vertex disjoint $(x,z)$-paths $L_1,\dots,L_{n-2}$ in $B_{n}[V_1\cup V_3]$.
Let $\widehat{x}=x[n-2,n-1]$ and $\widehat{z}=z[n-2,n-1]$.
Note that each $L_i$ contains exactly one vertex in $N_{B_{n-1}^1}(x)$. Assume that $\widehat{x}\in V(L_{n-2})$ and
$V(L_i)\cap N_{B_{n}}(x)=\{x_i\}$ for $i\in [n-3]$.
Similarly, we may assume that $\widehat{z}\in V(L_s)$ for some $s\in [n-2]$ and  $V(L_i)\cap N_{B_{n}}(z)=\{z_i\}$ for $i\in [n-2]\setminus\{s\}$.
Let $X=\{x_i': i\in [n-3]\}\cup \{x'\}$ and $Z=\{z_i': i\in [n-2]\setminus\{s\} \}\cup \{z'\}$.
Then $X\subseteq V_2$ with $|X|=n-2$ and $Z\subseteq V_4$ with $|Z|=n-2$.
By Lemmas \ref{menger1} and \ref{kappa}, there are $n-2$ internally vertex disjoint $(y,X)$-paths $Q_1,\dots, Q_{n-2}$ in $B_{n-1}^2$  and there are $n-2$ internally vertex disjoint $(w,Z)$-paths $R_1,\dots,R_{n-2}$ in $B_{n-1}^4$.
Assume that $x_i'\in V(Q_{i})$ for $i\in [n-3]$, $x'\in V(Q_{n-2})$ and $z'\in V(R_s)$, $z_i'\in V(R_i)$ for $i\in [n-2]\setminus\{s\}$.
If $s=n-2$, let
\[
T_i=L_i+x_ix_i'+Q_i+z_iz_i'+R_i \mbox{ for } i\in [n-3]
\]
and
\[
T_{n-2}=L_{n-2}+xx'+Q_{n-2}+zz'+R_{n-2}.
\]
Otherwise, let
\[
T_i=L_i+x_ix_i'+Q_i+z_iz_i'+R_i \mbox{ for } i\in [n-3]\setminus\{s\},
\]
\[
T_s=L_s+x_sx_s'+Q_s+zz'+R_s
\]
and
\[
T_{n-2}=L_{n-2}+xx'+Q_{n-2}+z_{n-2}z_{n-2}'+R_{n-2}.
\]
Then $T_1,\dots, T_{n-2}$ are $n-2$ internally edge disjoint trees  connecting $S$.
Next suppose $z'\in V_1\cup V_2$, say $z'\in V_1$.
There are $n-2$ internally vertex disjoint $(z,w)$-paths $L_1,\dots,L_{n-2}$ by Lemma \ref{ijout}.
Let $\widehat{z}=z[n-2,n-1]$.
Note that  each $L_i$ contains exactly one vertex in $N_{B_{n-1}^3}(z)$. Assume that $\widehat{z}\in V(L_{n-2})$ and $V(L_i)\cap N_{B_n}(z)=\{z_i\}$ for $i\in [n-3]$.
Let $Z=\{z_i': i\in [n-3]\}\cup \{z'\}$.
Then $Z\subseteq V_1$ with $|Z|=n-2$.
By Lemma \ref{kappa}, there are $n-2$ internally vertex disjoint $(x,Z)$-paths $Q_1,\dots,Q_{n-2}$.
Assume that $z_i'\in V(Q_i)$ for $i\in [n-3]$ and $z'\in V(Q_{n-2})$.
Let $\widehat{x}=x[n-2,n-1]$.
Note that  each $Q_i$ contains exactly one vertex in $N_{B_{n-1}^1}(x)$. Assume that $\widehat{x}\in V(Q_s)$ for some $s\in [n-2]$ and
$V(Q_i)\cap N_{B_{n-1}^1}(x)=\{y_i\}$ for $i\in [n-2]\setminus\{s\}$.
Let $X=\{x_i': i\in [n-2]\setminus\{s\}\}\cup \{x'\}$.
Then $X\subseteq V(B_{n-1}^2)$ with $|X|=n-2$.
There are $n-2$ internally vertex disjoint $(y,X)$-paths $R_1,\dots,R_{n-2}$ by Lemma \ref{kappa}.
Assume that $x_i'\in V(R_i)$ for $i\in [n-2]\setminus\{s\}$ and $x'\in V(R_s)$.
If $s=n-2$, let
\[
T_i=L_i+z_iz_i'+Q_i+x_ix_i'+R_i \mbox{ for }i\in [n-3]
\]
and
\[
T_{n-2}=L_{n-2}+zz'+Q_{n-2}+xx'+R_{n-2}.
\]
Otherwise, let
\[
T_i=L_i+z_iz_i'+Q_i+x_ix_i'+R_i \mbox{ for }i\in [n-3]\setminus\{s\},
\]
\[
T_s=L_s+z_sz_s'+Q_s+xx'+R_s
\]
and
\[
T_{n-2}=L_{n-2}+zz'+Q_{n-2}+x_{n-2}x_{n-2}'+R_{n-2}.
\]
Then there are $n-2$ internally edge disjoint trees $T_1,\dots,T_{n-2}$ connecting $S$.

\noindent
{\bf  Case 4.2.} $y'\notin V_1$.

Note that $y'\in V_3\cup V_4$. Assume that $y'\in V_3$.
By Lemma \ref{ijout}, there are $n-2$ internally vertex disjoint $(x,w)$-paths $L_1,\dots,L_{n-2}$ in $B_{n}[V_1\cup V_4]$.
Let $\widehat{x}=x[n-2,n-1]$.
Note that each $L_i$ contains exactly one vertex in $N_{B_{n-1}}(x)$. Assume that $\widehat{x}\in V(L_{n-2})$ and
$V(L_i)\cap N_{B_{n-1}^1}(x)=\{x_i\}$ for $i\in [n-3]$.
Let $X=\{x_i': i\in [n-3]\}\cup \{x'\}$.
Then $X\subseteq V_2$ with $|X|=n-2$.
By Lemmas \ref{menger1} and \ref{kappa}, there are $n-2$ internally vertex disjoint $(y,X)$-paths $Q_1,\dots, Q_{n-2}$ in $B_{n-1}^2$.
Assume that $x_i'\in V(Q_i)$ for $i\in [n-3]$ and $x'\in V(Q_{n-2})$.
Let $\widehat{y}=y[n-2,n-1]$.
Note that  each $Q_i$ contains exactly one vertex in $N_{B_{n-1}}(y)$. Assume that $\widehat{y}\in V(Q_{s})$ for some $s\in [n-2]$ and
$V(Q_i)\cap  N_{B_{n-1}}(y)=\{y_i\}$ for $i\in [n-2]\setminus\{s\}$.
Let $Y=\{y_i': i\in [n-2]\setminus\{s\}\}\cup \{y'\}$.
Then $Y\subseteq V_3$ with $|Y|=n-2$.
Since $\kappa(B_{n-1}^3)=n-2$, there are $n-2$ internally vertex disjoint $(z,Y)$-paths $R_1,\dots,R_{n-2}$ in $B_{n-1}^3$.
Assume that $y_i'\in V(R_i)$ for $i\in [n-2]\setminus\{s\}$ and $y'\in V(R_s)$.
If $s=n-2$, let
\[
T_i=L_i+x_ix_i'+Q_i+y_iy_i'+R_i\mbox{ for }i\in [n-3],
\]
and
\[
T_{n-2}=L_{n-2}+xx'+Q_{n-2}+yy'+R_{n-2}.
\]
Otherwise, let
\[
T_i=L_i+x_ix_i'+Q_i+y_iy_i'+R_i \mbox{ for } i\in [n-3]\setminus\{s\},
\]
\[
T_{s}=L_{s}+x_sx_s'+Q_{s}+yy'+R_{s}
\]
and
\[
T_{n-2}=L_{n-2}+xx'+Q_{n-2}+y_{n-2}y_{n-2}'+R_{n-2}.
\]
Then there are $n-2$ internally edge disjoint trees $T_1,\dots, T_{n-2}$ connecting $S$.

\subsection{Case 5}

Assume that $x,y,z\in V_1$ and  $w\in V_2$.

Suppose first that $n=4$.
Note that $B_3$ is a cycle of length $6$.
Let $P_{xy}$, $P_{xz}$, and $P_{yz}$ be the $(x,y)$-path, $(x,z)$-path and $(y,z)$-path in $B_3^1$ with $z\notin V(P_{xy})$, $y\notin V(P_{xz})$ and $x\notin V(P_{yz})$, respectively.
Suppose that $w'\in V_1$. If $w'\notin \{x,y,z\}$, then there is a spanning tree $T_1$  in $B_3^1$ and a spanning tree $T_2$ in $B_4[V(B_4)\setminus V_1]$, so
$T_1^*=T_1+w'w$ and
$T_2^*=T_2+x'x+y'y+z'z$ are two internally edge disjoint trees connecting $S$.
If $w'\in \{x,y,z \}$, say $w'=x$, then there is a spanning tree $T$ in $B_4[V(B_4)\setminus V_1]$,  so $T_1=wx+P_{xy}+P_{yz}$ and $T_2=P_{xz}+zz'+T+y'y$ are two internally edge disjoint trees connecting $S$.
Next suppose that $w'\notin V_1$.
Note that one of $x',y',z'$, say $x'$, lies outside $B_3^2$. Then $x'\in V_3\cup V_4$. Assume that $x'\in V_3$.
Let $x_1=x[1,2]$ and assume that $x_1\in V(P_{xy})$.

If $z'\in V_2$, then we choose a vertex $w_1$ in $B_3^2$ different from $w,z'$ such that $w_1'\in V_3$.
By Lemmas \ref{menger1} and \ref{kappa}, there are two $(w,\{w_1,z' \})$-paths $L_1$ and $L_2$.
Assume that $w_1\in V(L_1)$ and $z'\in V(L_2)$.
Since $w,w_1'\in V_3\cup V_4$, there are two $(\{w',w_1 '\},\{ x',x_1'\})$-paths $Q_1$ and $Q_2$ in $B_4[V_1\cup V_2]$ by Lemmas \ref{menger} and \ref{ijout}.
Assume that $w_1'\in V(Q_1)$.
If $x'\in V(Q_1)$, let $T_1=P_{yz}+P_{xz}+xx'+Q_1+w_1'w_1+L_1$ and $T_2=P_{xy}+x_1x_1'+Q_2+w'w+L_2+z'z$.
If $x'\in V(Q_2)$, let $T_1=P_{yz}+P_{xz}+xx'+Q_2+w'w$ and $T_2=P_{xy}+x_1x_1'+Q_1+w_1'w_1+L_1+L_2+z'z$.
Then $T_1$ and $T_2$ are two internally edge disjoint trees connecting $S$.

If $z'\in V_3\cup V_4$, say $z'\in V_4$. If $w'\in V_4$, let $w_1$ and $w_2$ be two vertices in $B_3^2$ with $w_1', w_2'\in V_3$, and there are two internally vertex disjoint $(w,w_i)$-path $L_i$ for $i=1,2$ in $B_3^2$ by Lemma \ref{kappa}.
Similarly, there are two internally vertex disjoint $(\{x',x_1'\},\{w_1',w_2'\} )$-paths $Q_1$ and $Q_2$ in $B_3^3$ and one $(w',z')$-path $K$ in $B_3^4$.
Then $T_1=P_{yz}+P_{xz}+xx'+Q_1+w_1'w_1+L_1$ and $T_2=P_{xy}+x_1x_1'+Q_2+w_2'w_2+L_2+ww'+K+z'z$
are two internally edge disjoint trees connecting $S$.
Otherwise, $w'\in V_3$.
Let $w_1$ and $w_2$ be two vertices in $B_3^2$ with $w_1'\in V_3$ and $w_2'\in V_4$.
By similar argument above, we may obtain two internally edge disjoint trees connecting $S$.

Now suppose that $n\ge 5$.
Let
\[x=(p_1,\dots, p_{n-1},1),y=(q_1,\dots, q_{n-1},1),z=(r_1,\dots, r_{n-1},1).
\]
Then $x\in V_{p_{n-1}}^1$, $y\in V_{q_{n-1}}^1$ and $z\in V_{r_{n-1}}^1$.

\noindent
{\bf Case  5.1.} $x',y'$ and $z'$ lie in three different main parts.

Let $x_i=x[i,i+1]$ for $i\in [n-2]$.
Then $x_1,\dots,x_{n-3}\in V_{p_{n-1}}^1$ and $x_{n-2}\in V_{p_{n-2}}^1$.
Since $q_{n-1}\neq r_{n-1}$, we may assume that $x_{n-2}\notin V_{q_{n-1}}^1$.
By Lemma \ref{ijout}, there are $n-3$ internally vertex disjoint $(x,y)$-paths $L_1,\dots,L_{n-3}$ in $B_{n-1}^1$.
Assume that $x_i\in V(L_i)$ for $i\in [n-3]$.
Let $\widehat{x}_i=x_i[n-2,n-1]$ for $i\in [n-4]$ and let $Z=\{\widehat{x}_i: i\in [n-4]\}\cup  \{x_{n-2}\}$.
We have $Z\subseteq V_1\setminus (V_{p_{n-1}}^1\cup V_{q_{n-1}}^1)$. As $\kappa(B_{n-1}^1[V_1\setminus (V_{p_{n-1}}^1\cup V_{q_{n-1}}^1)])=n-3$,
there are $n-3$ internally vertex disjoint $(z,Z)$-paths $Q_1,\dots,Q_{n-3}$.
Assume that $\widehat{x}_i\in V(Q_i)$ for $i\in [n-4]$ and $x_{n-2}\in V(Q_{n-3})$.
Let $F=\{x_i': i\in [n-3]\}\cup \{x',y',z'\}$ and $F_1=F\cap V_2$.

\noindent
{\bf Case  5.1.1.} $F_1=\emptyset$.

There are  three possibilities:
(i) $w'\notin V_1\cup V_{p_{n-1}}$, (ii) $w'\in V_{p_{n-1}}$ and (iii) $w'\in V_1$.

For (i), choose $n-2$ vertices $w_1,\dots, w_{n-2}\in V_2$ with out-neighbors in $V_{p_{n-1}}$.
Then there are $n-2$ internally vertex disjoint $(w,w_i)$-paths $H_i$ for $i\in [n-2]$ in $B_{n-1}^2$.
Let $X=\{x_i': i\in [n-3] \}\cup \{x'\}$ and  $W=\{w_i': i\in [n-2]\}$.
Then $X,W\subseteq V_{p_{n-1}}$ with $|X|=|W|=n-2$.
By Lemma \ref{menger}, there are $n-2$ internally vertex disjoint $(X,W)$-paths $R_1,\dots, R_{n-2}$ in $B_{n-1}^{p_{n-1}}$.
Assume that $x_i', w_i'\in V(R_i)$ for $i\in [n-3]$ and $x',w_{n-2}\in V(R_{n-2})$.
Since $y',z',w'\notin V_1\cup V_{p_{n-1}}$ and $B_n[V(B_{n})\setminus (V_1\cup V_{p_{n-1}})]$
is connected, there is a tree $T$ containing $y',z',w'$.
Let
\[
T_i=H_i+w_iw_i'+R_i+x_i'x_i+L_i+x_i\widehat{x}_i+Q_i\mbox{ for }i\in [n-4],
\]
\[
T_{n-3}=H_{n-3}+w_{n-3}w_{n-3}'+R_{n-3}+x_{n-3}'x_{n-3}+L_{n-3}+xx_{n-2}+Q_{n-3},
\]
and
\[
T_{n-2}=xx'+R_{n-2}+w_{n-2}'w_{n-2}+H_{n-2}+ww'+T+y'y+z'z
\]
Then there are $n-2$ internally edge disjoint trees $T_1,\dots,T_{n-2}$ connecting $S$.

For (ii), let $w_1,\dots, w_{n-3}$ be $n-3$ vertices in $B_{n-1}^2$ with out-neighbors in $B_{n-1}^{p_{n-1}}$ and $w_{n-2}\in V_2$ be one vertex with out-neighbor in $B_{n}[V(B_n)\setminus (V_1\cup V_{p_{n-1}})]$. By Lemmas \ref{menger1} and \ref{kappa},
there are $n-2$ internally vertex disjoint $(w,w_i)$-paths $H_i$ for $i\in [n-2]$ in $B_{n-1}^2$.
Let $X=\{x_i': i\in [n-3] \}\cup \{x'\}$ and  $W=\{w_i': i\in [n-3]\}\cup \{w'\}$.
Then $X,W\subseteq V_{p_{n-1}}$ with $|X|=|W|=n-2$.
By Lemma \ref{menger}, there are $n-2$ internally vertex disjoint $(X,W)$-paths $R_1,\dots, R_{n-2}$ in $B_{n-1}^{p_{n-1}}$.
Assume that  $x_i'\in V(R_i)$ for $i\in [n-3]$, $x'\in V(R_{n-2})$,  $w'\in V(R_s)$ for some $s\in [n-2]$, $w_i'\in V(R_i)$ for $i\in [n-3]\setminus\{s\}$ and $w_s'\in V(Q_{n-2})$.
Since $y',z',w_{n-2}'\notin V_1\cup V_{p_{n-1}}$, there is a tree $T$ with $y',z',w_{n-2}'\in V(T)$ in $B_n[V(B_n)\setminus(V_1\cup V_{p_{n-1}})]$.
If $s=n-2$, let
\[
T_i=H_i+w_iw_i'+R_i+x_i'x_i+L_i+x_i\widehat{x}_i+Q_i\mbox{ for } i\in [n-4],
\]
\[
T_{n-3}=H_{n-3}+w_{n-3}w_{n-3}'+R_{n-3}+x_{n-3}'x_{n-3}+L_{n-3}+xx_{n-2}+Q_{n-3},
\]
and
\[
T_{n-2}=xx'+R_{n-2}+w'w+H_{n-2}+w_{n-2}w_{n-2}'+T+y'y+z'z.
\]
If $s=n-3$, let
\[
T_i=H_i+w_iw_i'+R_i+x_i'x_i+L_i+x_i\widehat{x}_i+Q_i\mbox{ for } i\in [n-4],
\]
\[
T_{n-3}=ww'+R_{n-3}+x_{n-3}'x_{n-3}+L_{n-3}+xx_{n-2}+Q_{n-3}
\]
and
\[
T_{n-2}=xx'+R_{n-2}+w_{n-3}'w_{n-3}+H_{n-3}+H_{n-2}+w_{n-2}w_{n-2}'+T+y'y+z'z.
\]
Otherwise, let
\[
T_i=H_i+w_iw_i'+R_i+x_i'x_i+L_i+x_i\widehat{x}_i+Q_i\mbox{ for } i\in [n-4]\setminus\{s\},
\]
\[
T_s=ww'+R_s+x_s'x_s+L_s+x_s\widehat{x}_s+Q_s,
\]
\[
T_{n-3}=H_{n-3}+w_{n-3}w_{n-3}'+R_{n-3}+x_{n-3}'x_{n-3}+L_{n-3}+xx_{n-2}+Q_{n-3}
\]
and
\[
T_{n-2}=xx'+R_{n-2}+w_{s}'w_{s}+H_{s}+H_{n-2}+w_{n-2}w_{n-2}'+T+y'y+z'z.
\]
Then $T_1,\dots,T_{n-2}$ are $n-2$ internally edge disjoint trees connecting $S$.

Now we consider (iii). Suppose  that $N_{B_{n-1}^1}[w']\cap \cup_{i=1}^{n-3}(V(L_i)\cup V(Q_i))=\emptyset$. Let $\widehat{w}=w'[n-2,n-1]$.
If $\widehat{w}'\notin V_{p_{n-1}}$ ($\widehat{w}'\in V_{p_{n-1}}$, respectively),
then we use  $\widehat{w}'$ for $w'$ in the above argument in (i)  ((ii), respectively).
So we obtain $n-2$ internally edge disjoint trees connecting $S$.
Otherwise, assume that $\widetilde{w}\in N_{B_{n-1}^1}[w']\cap \cup_{i=1}^{n-3}(V(L_i)\cup V(Q_i))$.
Since $F_1=\emptyset$, $\widetilde{w}\in V(Q_s)$ for some $s\in [n-3]$.
Let $w_1,\dots,w_{n-3}$ be $n-3$ vertices in $B_{n-1}^2$ with out-neighbors in $B_{n-1}^{p_{n-1}}$ and $w_{n-2}$ be a vertex in $B_{n-1}^2$ with out-neighbor in $B_n[V(B_n)\setminus(V_1\cup V_{p_{n-1}})]$.
By Lemma \ref{ijout}, there are $n-2$ internally vertex disjoint $(w,w_i)$-paths $H_i$ for $i\in [n-2]$.
Let $X=\{x_i': i\in [n-3]\setminus\{s\}\}\cup \{x'\}$ and $W=\{w_i': i\in [n-3]\}$.
Then $X,W\subseteq V_{p_{n-1}}$ with $|X|=|W|=n-3$.
By Lemmas \ref{menger} and \ref{kappa}, there are $n-3$ internally vertex disjoint $(X,W)$-paths $R_1,\dots,R_{n-3}$ in $B_{n-1}^{p_{n-1}}$.
Assume that $x_i',w_i'\in V(R_i)$ for $i\in [n-3]\setminus\{s\}$ and $x',w_s'\in V(R_s)$.
Since $w_{n-2}',y',z'\notin V_1\cup V_{p_{n-1}}$, there is a spanning tree $T$ in $B_n[V(B_n)\setminus (V_1\cup V_{p_{n-1}})]$  with $w_{n-2}',y',z'\in V(T) $ by Lemma \ref{ijout}.
If $s=n-3$, let
\[
T_i=H_i+w_iw_i'+R_i+x_i'x_i+L_i+x_i\widehat{x}_i+Q_i\mbox{ for } i\in [n-4],
\]
\[
T_{n-3}=
\begin{cases}
L_{n-3}+xx_{n-2}+Q_{n-3}+w'w & \mbox{ if }\widetilde{w}=w'\\
L_{n-3}+xx_{n-2}+Q_{n-3}+\widetilde{w}w'+w'w&\mbox{ otherwise}
\end{cases}
\]
and
\[
T_{n-2}=xx'+R_{n-3}+w_{n-3}'w_{n-3}+H_{n-3}+H_{n-2}+w_{n-2}w_{n-2}'+T+y'y+z'z.
\]
Otherwise, let
\[
T_i=H_i+w_iw_i'+R_i+x_i'x_i+L_i+x_i\widehat{x}_i+Q_i\mbox{ for } i\in [n-4]\setminus\{s\},
\]
\[
T_s=
\begin{cases}
	L_s+x_s\widehat{x}_s+Q_s+w'w,&\mbox{ if }\widetilde{w}=w',\\
	L_s+x_s\widehat{x}_s+Q_s+\widetilde{w}w'+w'w&\mbox{ otherwise},
\end{cases}
\]
\[
T_{n-3}=H_{n-3}+w_{n-3}w_{n-3}'+R_{n-3}+x_{n-3}'x_{n-3}+L_{n-3}+xx_{n-2}+Q_{n-3},
\]
and
\[
T_{n-2}=xx'+R_{n-3}+w_{s}'w_{s}+H_{s}+H_{n-2}+w_{n-2}w_{n-2}'+T+y'y+z'z.
\]
Then $T_1,\dots,T_{n-2}$ are $n-2$ internally edge disjoint trees connecting $S$.

\noindent
{\bf Case  5.1.2.} $F_1=\{x_i': i\in[n-3]\}\cup \{x'\}$.

Suppose that $w'\notin V_1$. By Lemmas \ref{menger1} and \ref{kappa},
 there are $n-2$ internally vertex disjoint $(w,F_1)$-paths $H_1,\dots,H_{n-2}$ in $B_{n-1}^2$.
Assume that $x_i'\in V(H_i)$ for $i\in [n-3]$ and $x'\in V(H_{n-2})$.
Since $y',z',w'\in V(B_n)\setminus (V_1\cup V_2)$, there is a spanning  tree $T$ in $B_n[V(B_n)\setminus (V_1\cup V_2)]$ with $ y',z',w'\in V(T)$.
Let
\[
T_i=H_i+x_i'x_i+L_i+x_i\widehat{x}_i+Q_i\mbox{ for }i\in [n-4],
\]
\[
T_{n-3}=H_{n-3}+x_{n-3}'x_{n-3}+L_{n-3}+xx_{n-2}+Q_{n-3},
\]
and
\[
T_{n-2}=T+yy'+zz'+w'w+H_{n-2}+x'x.
\]
Then there are $n-2$  internally edge  disjoint trees $T_1,\dots,T_{n-2}$ connecting $S$.

Suppose that $w'\in V_1$.
If $N_{B_{n-1}^1}[w']\cap \cup_{i=1}^{n-3}(V(L_i)\cup V(Q_i))=\emptyset$, then
we may consider  $\widehat{w}'$ as $w'$ in the argument above with  $\widehat{w}=w'[n-2,n-1]$, and hence obtain $n-2$ internally edge disjoint  trees connecting $S$.
Otherwise,  some vertex in $N_{B_{n-1}^1}[w']$  lies on some path $L_i$ or $Q_i$, so the argument is similar to that in Case 5.1.1.

\noindent
{\bf Case  5.1.3.} $F_1=\{y'\}$ or $F_1=\{z'\}$, say $F_1=\{y'\}$.

Let $y_i=y[i,i+1]$ for $i\in [n-2]$.
Note $p_{n-1}\neq r_{n-1}$. Assume that $y_{n-2}\notin V_{p_{n-1}}^1$.
By Lemma \ref{ijout}, there are $n-3$ internally vertex disjoint $(x,y)$-paths $L_1,\dots,L_{n-3}$ in $B_{n-1}^1[V_{p_{n-1}}^1\cup V_{q_{n-1}}^1]$.
Assume that $y_i\in V(L_i)$ for $i\in [n-3]$.
Let $\widehat{y}_i=y_i[n-2,n-1]$ for $i\in [n-4]$ and let $Z=\{\widehat{y}_i:i\in [n-4] \}\cup \{y_{n-2}\}$.
Then $Z\subseteq V_1 \setminus (V_{p_{n-1}}^1\cup V_{q_{n-1}}^1)$.
Since $\kappa(B_{n-1}^1[V_1\setminus (V_{p_{n-1}}^1\cup V_{q_{n-1}}^1)])=n-3$, there are $n-3$ internally vertex disjoint $(z,Z)$-paths $Q_1,\dots,Q_{n-3}$.
Assume that $\widehat{y}_i\in V(Q_i)$ for $i\in [n-4]$ and $y_{n-2}\in V(Q_{n-3})$.
Let $F_2=(\{y_i':i\in[n-3]\}\cup \{x',y',z'\})\cap V_2$.
Recall that $y'\in V_2$, then $F_2=\{y_i':i\in[n-3]\}\cup \{y'\}$. Now by considering whether $w'$ is in $V_1$ and similar argument as in
  Case 5.1.2, there are $n-2$ internally edge disjoint trees connecting $S$.

\noindent
{\bf Case  5.2.} $x',y'$ and $z'$ lie in two different main parts.

Assume that $x'$ lies in different main part from $y'$ and $z'$. For $j\in [n]\setminus\{2\}$,

Since $\kappa(B_{n-2}^{(1,j)})=n-3$, there are $n-3$ internally vertex disjoint $(y,z)$-paths $L_1,\dots,L_{n-3}$ in $B_{n-2}^{(1,q_{n-1})}$.
Assume that $y_i\in V(L_i)$ and let $\widehat{y}_i=y_i[n-2,n-1]$ for $i\in [n-3]$.
Then $\widehat{y}_i\in( V_{q_{n-2}}^1\cup V_{q_{n-3}}^1)\subseteq V_1$.
Since $x\in V_{p_{n-1}}^1\subseteq V_1$, there are $n-3$ internally vertex disjoint $(x,\widehat{y}_i)$-paths $Q_i$ in $B_{n-1}^1[V_{p_{n-1}}\cup V_{q_{n-2}}\cup  V_{q_{n-3}}^1]$ for $i\in [n-3]$ by Lemma \ref{ijout}.
Let $x_i=x[i,i+1]$ for $i\in [n-3]$. Assume that $x_i\in V(Q_i)$ with  $i\in [n-3]$.
Let $F=\{x_i': i\in[n-3]\}\cup \{x',y',z'\}$ and $F_1=F\cap V_2$.
There are three possibilities:
(i) $F_1=\emptyset$,
(ii) $F_1=\{x_i': i\in [n-3]\}\cup \{x'\}$ and
(iii) $F_1=\{y',z'\}$.
The argument for  (i) and (ii) is similar as in  Case 5.1.1 and Case 5.1.2, respectively. So we only  consider (iii).
Suppose that $w\in V_\ell^2$ with $\ell\ne 2$.

\noindent
{\bf Case 5.2.1.} $w'\notin V_1$.

We choose $n-3$ vertices $w_1,\dots,w_{n-3}\in V_\ell^2$, then there are $n-3$ internally vertex disjoint $(w,w_i)$-paths $H_i$ for $i\in [n-3]$ in $B_{n-2}^{(2,\ell)}$.
Let $\widehat{w}=w[n-2,n-1]$. As $\widehat{w}, y',z'\in V_2\setminus V_\ell^2$, there is a tree $T_{n-2}^*$ containing $\widehat{w}, y',z'$ in $B_{n-2}^{(2,\ell)}$ by Lemma \ref{ijout}.
Let $Y=\{\widehat{y}_i': i\in [n-3]\}\cup \{x'\}$ and $W=\{w_i': i\in [n-3]\}\cup \{w'\}$.
Then $Y,W\subseteq V(B_n)\setminus(V_1\cup V_2)$, and so there are $n-2$ internally vertex disjoint $(Y,W)$-paths $R_1,\dots,R_{n-2}$ by Lemma \ref{ijout}.
Assume that $\widehat{y}_i'\in V(R_i)$ for $i\in [n-3]$, $x'\in V(R_{n-2})$, $w'\in V(R_s)$ for some $s\in [n-2]$, $w_i'\in V(R_i)$ for $i\in[n-3]\setminus\{s\}$ and $w_s'\in V(R_{n-2})$.
If $s=n-2$, let
\[
T_i=L_i+y_i\widehat{y}_i+Q_i+\widehat{y}_i\widehat{y}_i'+R_i+w_i'w_i+H_i\mbox{ for } i\in [n-3]
\]
and
\[
T_{n-2}=xx'+R_{n-2}+w'w+w\widehat{w}+T_{n-2}^*.
\]
Otherwise, let
\[
T_i=L_i+y_i\widehat{y}_i+Q_i+\widehat{y}_i\widehat{y}_i'+R_i+w_i'w_i+H_i\mbox{ for } i\in[n-3]\setminus\{s\},
\]
\[
T_s=L_s+y_s\widehat{y}_s+Q_s+\widehat{y}_s\widehat{y}_s'+R_s+ww'
\]
and
\[
T_{n-2}=xx'+R_{n-2}+w_s'w_s+H_s+w\widehat{w}+T_{n-2}^*.
\]
Hence, we obtain $n-2$ internally edge disjoint trees connecting $S$.

\noindent
{\bf Case 5.2.2.} $w'\in V_1$.

If $N_{B_{n-1}^1}[w']\cap \cup_{i=1}^{n-3}(V(L_i)\cup V(Q_i))=\emptyset$, the result follows by considering  $\widehat{w}'$ for  $w'$   in the above proof  with $\widehat{w}=w'[n-2,n-1]$.

Suppose that $N_{B_{n-1}^1}[w']\cap \cup_{i=1}^{n-3}(V(L_i)\cup V(Q_i))\not=\emptyset$.
Let $y_i=y[i,i+1]$, $w_i=w[i,i+1]$ and assume that $y_i\in V(L_i)$, $\widehat{y}_i\in V(Q_i)$ for $i\in [n-3]$.
Let $\widehat{x}=x[n-2,n-1]$, $\widehat{z}=z'[n-2,n-1]$, $\widehat{w}=w[n-2,n-1]$ and $\widehat{w}_i=w_i[n-2,n-1]$ for $i\in [n-3]$.

Suppose that $y$ or $z$, say $y$, is adjacent to $w$. Then $y_i'=w_i$ for $i\in [n-3]$.
If $z$ is not adjacent to $y$, then, since $\kappa(B_{n-1}^2)=n-2$, there is a $(z',\widehat{w})$-path $R$ in $B_{n-1}^2[V_2\setminus \{w_i: i\in [n-3]\}]$.
Noting that $\widehat{w}',x'\in V(B_n)\setminus(V_1\cup V_2)$, there is an $(x',\widehat{w}')$-path $K$.
Let
\[
T_i=ww_i+w_i'y_i+L_i+y_i\widehat{y}_i+Q_i\mbox{ for } \in [n-3]
\]
and
\[
T_{n-2}=xx'+K+\widehat{w}'\widehat{w}+R+z'z+\widehat{w}w+wy.
\]
Then we obtain $n-2$ internally edge disjoint trees $T_1, \dots, T_{n-2}$ connecting $S$.
Suppose that $z$ is adjacent to $y$, say $z=y_{\xi}$ for some $\xi \in [n-2]$. Then $z'=w_\xi$.
Let
\[
T_i=ww_i+w_iy_i+L_i+y_i\widehat{y}_i+Q_i\mbox{ for } i\in[n-3]\setminus\{\xi\}.
\]
and $\widehat{y}=y[n-2,n-1]$.
We consider  $\widehat{y}\neq x$ and  $\widehat{y}= x$ separately.
Suppose that $\widehat{y}\neq x$.
Note that $x',\widehat{y}',\widehat{w}_\xi',\widehat{w}'\notin V_1\cup V_2$, there are two $(\{x',\widehat{y}'\},\{\widehat{w}_\xi',\widehat{w}'\})$-paths $R_1$ and $R_2$ by Lemma \ref{ijout}.
If $x'$ and $\widehat{w}_\xi'$ are in the same path, say $R_1$, let
\[
T_\xi=xx'+R_1+\widehat{w}_\xi'\widehat{w}_\xi+\widehat{w}_\xi w_\xi+w_\xi w+ wy,
\]
and
\[
T_{n-2}=Q_\xi+\widehat{y}y+yz+\widehat{y}\widehat{y}'+R_2+\widehat{w}'\widehat{w}+\widehat{w}w.
\]
Otherwise, assume that $x'$ is in $R_1$, let
\[
T_\xi=xx'+R_1+\widehat{w}'\widehat{w}+\widehat{w}w+wy+yz
\]
and
\[
T_{n-2}=Q_\xi+\widehat{y}y+\widehat{y}\widehat{y}'+R_2+\widehat{w}_\xi'\widehat{w}_\xi+\widehat{w}_\xi w_\xi+w_\xi w+w_\xi z.
\]
So we obtain $n-2$ internally edge disjoint trees $T_1, \dots, T_{n-2}$ connecting $S$.
Now suppose that $\widehat{y}=x$.
Then
\[
x=(q_1,\dots,q_{n-3},2,q_{n-2},1) \mbox{ and }  x'=(q_1,\dots,q_{n-3},2,1,q_{n-2}).
\]
Recall that $w=y'$.
Then
\[
w=(q_1,\dots,q_{n-3},q_{n-2},1,2), \widehat{w}=(q_1,\dots,q_{n-3},1,q_{n-2},2)
\]
and
\[
\widehat{w}'=(q_1,\dots,q_{n-3},1,2,q_{n-2}).
\]
It can be seen that $x'$ is adjacent to $\widehat{w}'$.
Let
\[
T_\xi=yw+wz'+z'z+z\widehat{y}_\xi+Q_\xi
\]
and
\[
T_{n-2}=zy+yx+xx'+x'\widehat{w}'+\widehat{w}'\widehat{w}+\widehat{w}w.
\]
So there are $n-2$ internally edge disjoint trees $T_1,\dots,T_{n-2}$ connecting $S$.

Next suppose that $y$ and $z$ are not adjacent to $w$.
Suppose that $y'$ or $z'$, say $y'$, is adjacent to $w$. Then $y_\xi=w'$ for some $\xi\in [n-3]$. So
\[
T_\xi=wy_\xi +L_\xi+y_\xi \widehat{y}_\xi+Q_\xi
\]
is a tree containing the vertices in $S$.
Let $W$ be the set of $n-4$ neighbors of $w$ such that they are not adjacent to $y$ or $z$ and $Y=\{\widehat{y}_i': i\in [n-3]\setminus\{\xi\}\}$.
Similarly to the above argument, we may obtain $n-4$ internally vertex disjoint $(Y,W)$-paths and hence  $n-4$ internally edge disjoint trees $T_i$ for $i\in [n-3]\setminus\{\xi\}$ connecting $S$.
Since $\kappa(B_{n-2}^{(2,1)})=n-3$, there is a tree $H$ containing $w,y',z'$.
Let $v_1=y'[n-2,n-1]$, $v_2=v_1[n-3,n-2]$,  $v_3=v_2[n-4,n-3]$,
$v_4=v_3[n-3,n-2]$, $v_5=v_4[n-2,n-1]$
and $P_y=wy'v_1v_2v_3v_4v_5$.
Note that there is an $(x',v_5')$-path $L_{n-2}$ with  $V(L_{n-2})\cap V(T_i)=\emptyset$ for $i\in [n-3]$.
Let
\[
T_{n-2}=H+P_y+v_5v_5'+L_{n-2}+x'x.
\]
So there are $n-2$ internally edge disjoint trees $T_1,\dots,T_{n-2}$ connecting $S$.
Suppose that $y',z'\notin\{w_i: i\in [n-3]\}\cup \{w\}$.
Let $Y=\{\widehat{y}_i': i\in [n-3]\}$ and $W=\{\widehat{w}_i': i\in [n-4]\}\cup \{w_{n-2}'\}$.
By Lemmas \ref{menger} and \ref{ijout}, there are $n-3$ internally vertex disjoint $(Y,W)$-paths in $B_n[V(B_n)\setminus (V_1\cup V_2)]$. Hence  we may obtain $n-3$ internally edge disjoint trees $T_i$ for $i\in [n-3]$ by similar argument as in Case 5.2.1.
Since $\kappa(B_{n-2}^{(2,1)})=n-3$, there is a tree $H$ in $B_{n-2}^{(2,1)}$ containing vertices $y',z',w$ with $V(H)\cap \{w_i: i\in[n-4]\}=\emptyset$.
By Lemma \ref{ijout}, there is an $(x',\widehat{w}_{n-3}')$-path $L_{n-2}$ such that it is disjoint with the above $n-3$ $(Y,W)$-paths.
Let
\[
T_{n-2}=yy'+zz'+H+w_{n-3}\widehat{w}_{n-3}+\widehat{w}_{n-3}\widehat{w}_{n-3}'+L_{n-2}+x'x.\]
Then there are $n-2$ internally edge disjoint trees $T_1,\dots,T_{n-2}$ connecting $S$.

\noindent
{\bf Case  5.3.} $x',y',z'$ lie in the same main part, that is, $p_{n-1}=q_{n-1}=r_{n-1}$.

\noindent
{\bf Case  5.3.1.}  There is at least one of $x,y,z$, say $x$, that is not adjacent to the others.

Since $x,y,z\in V_{p_{n-1}}^1$, there are $n-4$ internally edge disjoint trees $T_1,\dots,T_{n-4}$ connecting $\{x,y,z\}$ in $B_{n-2}^{(1,p_{n-1})}$ by Lemma \ref{kappa3}.
Note that each $T_i$ contains at least one vertex in $N_{B_{n-2}^{(1,p_{n-1})}}(x)$, say $x_i$, for $i\in [n-4]$.
Assume that $x_1=x[n-3,n-2]$.
Let $\widehat{x}_i=x_i[n-2,n-1]$ for $i\in [n-4]$,
$\widehat{x}=x[n-2,n-1]$,  $\widehat{y}=y[n-2,n-1]$ and $\widehat{z}=z[n-2,n-1]$.
Note that $\widehat{x},\widehat{y},\widehat{z}\notin V_{p_{n-1}}^1$ and $\kappa(B_{n-2})=n-3$, there is a tree $T_{n-3}$ not in $B_{n-2}^{(1,p_{n-1})}$ containing  $\widehat{x},\widehat{y},\widehat{z}$ with $V(T_{n-3})\cap\{\widehat{x}_i: i=1,\dots,n-4\}=\emptyset$.
Assume that $x_1=x[n-3,n-2]$.

Let $F=\{\widehat{x}_i': i=1,\dots,n-4\}\cup\{\widehat{x}',x',y',z'\}$ and $F_1=F\cap V_2$.
Note that $\widehat{x}_1'\in V_{p_{n-3}}$,
$\widehat{x}_i'\in V_{p_{n-2}}$ for $i=2,\dots,n-4$, $\widehat{x}'\in V_{p_{n-2}}$ and $x',y',z'\in V_{p_{n-1}}$.
There are four possibilities:
(i) $F_1=\emptyset$,
(ii) $F_1=\{\widehat{x}_i: i=2,\dots,n-4\}\cup \{\widehat{x}\}$,
(iii) $F_1=\{\widehat{x}_1\}$, and
(iv) $F_1=\{x',y',z'\}$.
Note that (i)--(iii) can be discussed similarly as in Case 5.1.
Then we only need to consider  (iv).

If $w'\notin V_1$, then $w\notin V_1^2$ and $x',y',z'\in V_1^2$, and so the result follows by similar argument as in Case 5.2.
So we assume that $w'\in V_1$.

Suppose first that $x',y',z'\notin N_{B_{n-1}^2}[w]$.
Let $w_i=w[i,i+1]$, $\widehat{w}_i=w_i[n-2,n-1]$ for $i\in [n-4]$ and $\widehat{w}=w[n-2,n-1]$.
Then $\widehat{w}_i'\notin V_1$ for $i\in [n-4]$.
Let
$W=\{\widehat{w}_i': i\in [n-4]\}\cup \{\widehat{w}'\}$
and $X=\{\widehat{x}_i: i\in [n-4]\}\cup \{\widehat{x}'\}$.
By Lemma \ref{ijout}, there are $n-3$ internally vertex disjoint $(X,W)$-paths $L_1,\dots,L_{n-3}$ in $B_n[V(B_n)\setminus (V_1\cup V_2)]$.
Assume that $\widehat{x}_i'\in V(L_i)$ for $i\in [n-4]$, $\widehat{x}'\in V(L_{n-3})$, $\widehat{w}\in V(L_s)$ for some $s\in [n-3]$,  $\widehat{w}_i'\in V(L_i)$ for $i\in [n-4]\setminus\{s\}$  and $\widehat{w}_s\in L_{n-3}$.
If $s=n-3$,
let
\[
T_i^*=T_i+x_i\widehat{x}_i+\widehat{x}_i\widehat{x}_i'+L_i+\widehat{w}_i'\widehat{w}_i+\widehat{w}_iw_i+w_iw\mbox{ for }i\in [n-3]
\]
and
\[
T_{n-3}^*=T_{n-3}+x\widehat{x}+\widehat{x}\widehat{x}'+L_{n-3}+\widehat{w}'\widehat{w}+\widehat{w}w.
\]
Otherwise, let
\[
T_i^*=T_i+x_i\widehat{x}_i+\widehat{x}_i\widehat{x}_i'+L_i+\widehat{w}_i'\widehat{w}_i+\widehat{w}_iw_i+w_iw\mbox{ for }i\in[n-4]\setminus\{s\},
\]
\[
T_s^*=T_s+x_s\widehat{x}_s+\widehat{x}_s\widehat{x}_s'+L_s+\widehat{w}'\widehat{w}+\widehat{w}w,
\]
and
\[
T_{n-3}^*=T_{n-3}+x\widehat{x}+\widehat{x}\widehat{x}'+L_{n-3}+\widehat{w}_s'\widehat{w}_s+\widehat{w}_s w_s+w_sw.
\]
In $B_{n-2}^{(2,1)}$, there is a tree $T_{n-2}$ containing $w,x',y',z'$ with $V(T_{n-2})\cap \{w_i: i\in [n-4]\}=\emptyset$.
Let
\[
T_{n-2}^*=xx'+yy'+zz'+T_{n-2}.
\]
Hence, we obtain $n-2$ internally edge disjoint trees $T_1^*, \dots, T_{n-2}^*$ connecting $S$.

Suppose next that $\{x',y',z'\}\cap N_{B_{n-1}^2}[w]\neq \emptyset$.

Suppose that  $x'=w$, that is, $w$ is adjacent to $x$.
Then  $y'$ and $z'$ are not adjacent to $w$.
So $\{x_i': i\in [n-4]\}\cup \{x\}\subseteq N_{B_{n-1}^2}(w)$. Let $\widehat{w}=w[n-2,n-1]$.
Then  $\widehat{w}'$ is adjacent to $\widehat{x}'$.
Let
\[
T_i^*=T_i+x_ix_i'+x_i'w\mbox{ for } i\in [n-4]
\]
and
\[
T_{n-3}^*=T_{n-3}+x\widehat{x}+\widehat{x}\widehat{x}'+\widehat{x}'\widehat{w}'+\widehat{w}'\widehat{w}+\widehat{w}w.
\]
Since $y',z'\in V_2$ and $\kappa(B_{n-1}^2)=n-2$, there is a tree $T_{n-2}$ containing $w,y',z'$ in $B_{n-1}^2[V_2\setminus(\{x_i': i\in [n-4]\}\cup \{\widehat{w} \})]$.
Let
\[
T_{n-2}^*=zz'+yy'+T_{n-2}+wx.
\]
Then $T_1^*,\dots,T_{n-2}^*$ are $n-2$ internally edge disjoint trees connecting $S$.

Suppose  that $w$ is not adjacent to $x$.
Suppose that $y$ or $z$, say $y$, is adjacent to $w$. Then $x'$ is not adjacent to $w$.
Choose $n-4$ neighbors of $w$, say $w_1,\dots,w_{n-4}$ such that each of them is not adjacent to $y$ or $z$.
By similar proof when $x',y',z'\notin N_{B_{n-1}^2}[w]$, we may obtain $n-2$ internally edge disjoint trees connecting $S$. So assume in the following that $w$ is not adjacent to $y$ or $z$.
Let $w_i=w[i,i+1]$, $\widehat{w}_i=w_i[n-2,n-1]$ for $i\in [n-3]$, $\widehat{w}=w[n-2,n-1]$  and  $\widehat{x}_{n-3}=\widehat{x}$. Suppose first that there is exactly one  of $x',y',z'$, say $x'$, that is adjacent to $w$. Then $w'=x_s$ and $x'=w_t$ for some $s,t\in [n-3]$,
\[
T_s^*=T_s+xw
\]
is a tree containing vertices in $S$.
Let
$X=\{\widehat{x}_i': i\in [n-3]\setminus\{s\}\}$ and $W=\{\widehat{w}_i: i\in [n-3]\setminus\{t\}\}$.
By Lemmas \ref{menger} and \ref{ijout}, there are $n-4$ internally vertex disjoint $(X,W)$-paths $Q_i$ for $i\in [n-3]\setminus\{s\}$ in $B_n[V(B_n)\setminus (V_1\cup V_2)]$.
Assume that $\widehat{x}_i',\widehat{w}_{\ell_i}\in V(Q_i)$ for $i\in[n-3]\setminus\{s\}$, where $\ell_i\in [n-3]\setminus\{t\}$ and $\ell_i\neq \ell_j$ if $i\neq j$.
Let
\[
T_i^*=T_i+x_i\widehat{x}_i+\widehat{x}_i\widehat{x}_i'+Q_i+\widehat{w}_{\ell_i}'\widehat{w}_{\ell_i}+\widehat{w}_{\ell_i}w_{\ell_i}+w_{\ell_i}w
\]
for $i\in [n-4]\setminus\{s\}$
and
\[
T_{n-3}^*=T_{n-3}+x\widehat{x}+\widehat{x}\widehat{x}'+L_{n-3}+\widehat{w}_{\ell_{n-3}}'\widehat{w}_{\ell_{n-3}}+\widehat{w}_{\ell_{n-3}}w_{\ell_{n-3}}+w_{\ell_{n-3}}w.
\]
Since $\kappa(B_{n-2}^{(2,1)})=n-3$, there is a tree $T_{n-2}$ containing $w,x',y',z'$ with $V(T_{n-2})\cap \{w_i: i\in [n-3]\setminus\{t\}\}=\emptyset$ in $B_{n-2}^{(2,1)}$.
Let
\[
T_{n-2}^*=xx'+yy'+zz'+T_{n-2}.
\]
Hence, there are $n-2$ internally edge disjoint trees $T_1^*, \dots, T_{n-2}^*$ connecting $S$.
Suppose now that there are exactly two of $x',y',z'$, say $x'$ and $y'$, that are adjacent to $w$.
That is, $x'=w_t$ and $y'=w_r$ for some $t,r\in [n-3]$.
Assume that $w'=x_s$. Then
\[
T_s^*=T_s+xw
\]
is a tree containing vertices in $S$.
Let $\widehat{w}_{n-2}=w[n-2,n-1]$.
Let $X=\{\widehat{x}_i': i\in [n-3]\setminus\{s\}\}$ and $W=\{\widehat{w}_i': i\in [n-2]\setminus\{t,r\}\}$.
Then $X,W\subseteq V(B_n)\setminus (V_1\cup V_2)$, there are $n-4$ internally vertex disjoint $(X,W)$-paths $Q_i$ for $i\in [n-3]\setminus\{s\}$.
By similar argument as above, we may construct  $n-2$ internally edge disjoint trees  (one of which is $T_s^*$)  connecting $S$.
Finally suppose that $x',y',z'$ are all adjacent to $w$.
Then  there are some $t,r,s\in [n-3]$ with $t<r<s$ such that $x'=w_t$, $y'=w_r$ and $z'=w_s$.
Since $3\le s\le n-3$, $n\ge 6$.
Assume that $w'=x_\gamma$ for some $\gamma\in [n-4]$.
Then
$T_\gamma^*=T_\gamma+w'w$
 is a tree containing vertices in $S$.
It can be verified that $\widehat{x}'$ is adjacent to $\widehat{w}'$.
Let $X=\{\widehat{x}_i': i\in [n-4]\setminus\{\gamma\}\}$ and $W=\{\widehat{w}_i': i\in [n-3]\setminus\{t,r\}\}$.
By Lemma \ref{ijout}, there are $n-5$ internally vertex disjoint $(X,W)$-paths $L_i$ for $i\in [n-4]\setminus\{\gamma\}$ in $B_n[V(B_n)\setminus(V_1\cup V_2)]$ with $V(L_i)\cap \{\widehat{x}',\widehat{w}'\}=\emptyset$ for $i\in [n-5]$.
Assume that $\widehat{w}_s',\widehat{x}_\xi'\in V(L_\xi)$ for some $\xi\in [n-4]\setminus\{\gamma\}$ and $\widehat{x}_i',\widehat{w}_{\ell_i}'\in V(L_i)$ for $i\in [n-4]\setminus\{\gamma,\xi\}$, where $\ell_i\in [n-3]\setminus\{t,r,s\}$ and $\ell_i\neq \ell_j$ if $i\neq j$.
Let $y_\xi, z_\xi$ be the neighbors of $y$ and $z$ in $V(L_\xi)$, respectively.
Then $y_\xi',z_\xi'\in V_1^2$.
Let $\widehat{y}_\xi=y_\xi'[n-2,n-1]$.
Recall that $z'=\widehat{w}_s$  and $\kappa(B_{n-2})=n-3$.
So there is a $(\widehat{y}_\xi,\widehat{w}_s)$-path $Q_\xi$ in $B_{n-1}^2[V_2\setminus V_1^2]$ with $V(Q_\xi)\cap (\{\widehat{w}_i: i\in [n-3]\setminus\{t,r,s\}\}\cup \{\widehat{w}\})=\emptyset$.
By Lemma \ref{kappa}, there is a tree $T_{n-2}$ containing $w,x',y',z_\xi'$ in $B_{n-2}^{(2,1)}$ with $V(T_{n-2})\cap \{w_i: i\in [n-3]\setminus\{t,r\}\}=\emptyset$.
Let
\[
T_i^*=T_i+x_i\widehat{x}_i'+\widehat{x}_i\widehat{x}_i'+L_i+\widehat{w}_{\ell_i}'\widehat{w}_{\ell_i}
+\widehat{w}_{\ell_i}w_{\ell_i}+w_{\ell_i}w
\]
for $i\in [n-4]\setminus\{\gamma,\xi\}$,
\[
T_{n-3}^*=T_{n-3}+x\widehat{x}+\widehat{x}\widehat{x}'+\widehat{x}'\widehat{w}'+\widehat{w}'\widehat{w}+\widehat{w}w,
\]
\[
T_\xi^*=xx_\xi+x_\xi\widehat{x}_\xi+\widehat{x}_\xi\widehat{x}_\xi'+L_\xi+\widehat{w}_s'\widehat{w}_s+\widehat{w}_sw+\widehat{w}_sz
\]
and
\[
T_{n-2}^*=xx'+yy'+T_{n-2}+z_\xi'z_\xi+z_\xi z.
\]
Hence, we obtain $n-2$ internally edge disjoint trees $T_1^*, \dots, T_{n-2}^*$ connecting $S$.

\noindent
{\bf Case  5.3.2.} There is one of $x,y,z$, say $x$, is adjacent to the others.

Let $x_i=x[i,i+1]$ for $i\in [n-2]$. There exist $\ell,s\in [n-3]$ such that $y=x_\ell$ and $z=x_s$, where $\ell<s$.
Suppose first that $s=n-3$. For $i,j\in [n]\setminus\{1\}$ with $i\neq j$, let
$V_j^{1,i}=\{(v_1,\dots,v_{n-3},j,i,1): (v_1,\dots,v_{n-3})\in \mbox{Sym}_{1,i,j}(n)\}$,
where $\mbox{Sym}_{1,i,j}(n)$ is the set of permutations of $[n]\setminus\{1,i,j\}$.
Let $B_{n-3}^{(1,p_{n-1},j)}=B_{n}[V_j^{1,p_{n-1}}]$.
Then  $B_{n-3}^{(1,p_{n-1},j)}\cong B_{n-3}$.
Note that  $\kappa(B_{n-3})=n-4$.
So there are $n-4$ internally vertex disjoint $(x,y)$-paths $L_1,\dots,L_{n-4}$ in $B_{n-3}^{(1,p_{n-1},j)}$.
Assume that $x_i\in V(L_i)$ for $i\in [n-4]$.
Let $\widehat{x}_i=x_i[n-3,n-2]$.
Then $\widehat{x}_i\in V_{p_{n-1}}^1\setminus V_{p_{n-2}}^{(1,p_{n-1})}$, and there are $n-4$ internally vertex disjoint $(z,\widehat{x}_i)$-paths $Q_i$ for $i\in [n-4]$.
Let $y_{n-2}=y[n-2,n-1]$ and $z_{n-2}=z[n-2,n-1]$.
Then $x_{n-2},y_{n-2},z_{n-2}\in V_1\setminus V_{p_{n-1}}^1$, and there is a tree $T$ containing $x_{n-2},y_{n-2},z_{n-2}$ in $B_n[V_1\setminus V_{p_{n-1}^1}]$.

Let $F=\{\widehat{x}_i': i\in [n-4]\}\cup\{x_{n-2}',x',y'\}$ and $F_1=F\cap V_2$.
Then there are three possibilities:
(i) $F_1=\emptyset$,
(ii) $F_1=\{\widehat{x}_i': i\in [n-4]\}\cup \{x',y'\}$,
and (iii) $F_1=\{x_{n-2}'\}$.
By considering whether the out-neighbor of $w$ lies in $V_1$, and by similar discussions as in Case 5.1, we have $n-2$ internally edge disjoint trees $T_1, \dots, T_{n-2}$ such that $T_1, \dots, T_{n-3}$ connect $S$ and $T_{n-2}$ contains $x,y,w$. Let $T_{n-2}^*=T_{n-2}+xz+xx'+yy'$. Then
$T_1, \dots ,T_{n-3}$, $T_{n-2}^*$ are  $n-2$ internally edge disjoint trees  connecting $S$.


Suppose that $s<n-3$.
For any $t$ with $s\le t\le n-3$, let $\{i_1,\dots,i_{n-(t+2)},j\}\subset [n]\setminus \{1\}$.
Let
\begin{align*}
& V_j^{1,i_1,\dots,i_{n-(t+2)}}\\
= & \{(v_1,\dots,v_t,j,i_{n-(t+2)},\dots,i_1,1):(v_1,\dots,v_t)\in \mbox{Sym}_{1,i_1,\dots,i_{n-(t+2)},j}(n) \},
\end{align*}
where $\mbox{Sym}_{1,i_1,\dots,i_{n-(t+2)},j}(n) $ is the set of permutations of $[n]\setminus \{ 1,i_1,\dots,i_{n-(t+2)},j\}$.
Then $x,y\in V_{p_{s+1}}^{1,p_{n-1},\dots,p_{s+2}}$.
Since $B_n[V_{p_{s+1}}^{1,p_{n-1},\dots,p_{s+2}}]\cong B_s$ and
$\kappa(B_s)=s-1$, there are $s-1$ internally vertex disjoint $(x,y)$-paths $L_1,\dots,L_{s-1}$ in $B_n[V_{p_{s+1}}^{1,p_{n-1},\dots,p_{s+2}}]$.
Assume that $x_i\in V(L_i)$ for $i\in [s-1]$.
Let $\widehat{x}_i=x_i[s,s+1]$ for $i\in [s-1]$.
Then $z,\widehat{x}_i\in V_{p_{s+2}}^{1,p_{n-1},\dots,p_{s+3}}\setminus V_{p_{s+1}}^{1,p_{n-1},\dots,p_{s+2}}$.
By Lemma \ref{ijout}, there are $s-1$ internally vertex disjoint $(z,\widehat{x}_i)$-path $Q_i$ for $i\in [s-1]$.
Let $y_i=y[i,i+1]$ and $z_i=z[i,i+1]$ for $i=s+2,\dots,n-2$.
Since $B_n[V_{p_{i+1}}^{1,p_{n-1},\dots,p_{i+2}}\setminus V_{p_{i}}^{1,p_{n-1},\dots,p_{i+1}}]$ is connected, there is a tree $T^*_i$ containing $x_i,y_i,z_i $ for $i=s+2,\dots,n-2$.

Let $F=\{\widehat{x}_i':i\in[s-1]\}\cup \{x_i':i=s+2,\dots,n-2\}\cup \{x',y'\}$.
Then there are three possibilities:
(i) $F_1=\emptyset$,
(ii) $F_1=\{\widehat{x}_i':i\in[s-1]\}\cup \{x_i':i=s+2,\dots,n-3\}\cup \{x',y'\}$, and
(iii) $F_1=\{ x_{n-2}'\}$.
By considering whether the out-neighbor of $w$ lies in $V_1$, and similar discussions as in Case 5.1, we may have $n-2$ internally edge disjoint trees $T_1, \dots, T_{n-2}$ such that $T_1, \dots, T_{n-3}$  connect $S$ and one $T_{n-2}$ contains  $x,y,w$.
Let $T_{n-2}^*=T+xz+xx'+yy'$, we obtain $n-2$ internally edge disjoint trees $T_1, \dots, T_{n-3}$, $T_{n-2}^*$ connecting $S$.
\end{proof}

\section{Concluding remarks}

From a theoretical perspective, the generalized $k$-connectivity $\kappa_k(G)$ of a connected graph of order $n\ge 2$ includes two fundamental concepts:
the connectivity for $k=2$ and  the maximum number of edge disjoint spanning trees  for $k=n$.
From a practical perspective, the generalized connectivity can measure the reliability and security of a network.
The bubble-sort graph $B_n$  is a particular  Cayley graph that is suitable as a topology for massively parallel systems.
In this article, we prove that $\kappa_4(B_n)=n-2$ for $n\ge 3$. In other words,
there are $n-2$ internally disjoint trees connecting them in $B_n$  for any four vertices of $B_n$ when $n\ge 3$.
For further work,  it would be interesting to study the generalized connectivity of Cayley graphs on symmetric groups generated by general trees and some other important networks \cite{Xu}.

\vspace{5mm}

\noindent {\bf Acknowledgement.} 
This work was supported by National Natural Science Foundation of China (No. 12071158).

\end{document}